\documentclass[12pt]{amsart}
\usepackage{geometry}                
\usepackage{graphicx}
\usepackage{amssymb}
\usepackage{enumerate}
\usepackage{epstopdf}
\newtheorem{theorem}{Theorem}[section]
\newtheorem{lemma}[theorem]{Lemma} 
\newtheorem{remark}[theorem]{Remark} 
\newtheorem{definition}[theorem]{Definition} 
\newtheorem{corollary}[theorem]{Corollary} 
\newtheorem{claim}[theorem]{Claim} 
\newtheorem{example}[theorem]{Example}

\def\A{\mathcal{A}}
\def\B{\mathcal{B}}
\def\C{\mathcal{C}}
\def\D{\mathcal{D}}
\def\K{\mathcal{K}}

\def\a{\alpha}
\def\b{\beta}
\def\g{\gamma}
\def\o{\omega}
\def\k{\kappa}
\def\l{\lambda}
\def\t{t_{qf}}
\def\raj{\upharpoonright}
\def\v{\vert}

\title{Categoricity and universal classes}
\author{Tapani Hyttinen and Kaisa Kangas}

\begin{document}

\thanks{Research of the second author was supported by grant 310737 of the Academy of Finland}
\maketitle

\begin{center}
Department of Mathematics, University of Helsinki \\
P.O. Box 68, 00014, Finland \\
kaisa.kangas@helsinki.fi
\vspace{5mm}
\end{center}

\begin{abstract}
Let $(\K ,\subseteq )$ be a universal class with $LS(\K)=\l$
categorical in regular $\k >\l^{+}$ with arbitrarily large models, and let
$\K^{*}$ be the class of all $\A\in\K_{>\l}$ for which
there is $\B\in\K_{\ge\k}$ such that $\A\subseteq\B$.
We prove that $\K^{*}$ is totally categorical (i.e. $\xi$-categorical for all $\xi>LS(\K)$),
$\K_{\ge\beth_{(2^{\l^{+}})^{+}}}\subseteq\K^{*}$,
and the models of $\K^{*}_{>\l^{+}}$ are essentially vector spaces
(or trivial i.e. disintegrated).
\end{abstract}

\section{Introduction}

\noindent
Universal classes were introduced by S. Shelah in the 80's:
An abstract elementary class $(\K ,\preceq )$
is universal if $\preceq$ is the submodel relation $\subseteq$
and for all $\A\subseteq\B\in\K$, $\A\in\K$. 
In \cite{va}, S. Vasey started the study of categoricity transfer in universal classes. 
He showed that if $\K$ is a universal class and for all $\xi$ there
is $\xi'$ such that $cf(\xi')>\xi$ and $\K$ is categorical in $\xi'$,
then there is $\xi$ such that $\K$ is categorical in every $\xi'>\xi$.

This summer, while visiting Helsinki, Jonathan Kirby asked us whether a version of the following statement is true:

\vspace{0.2cm}

\noindent
{\it Suppose $\K$ is a universal class
with $LS(\K )=\l$ and categorical in some $\k >\l$.
Then $\K$ is categorical in every $\xi >\l$ and the models of $\K$ are either vector spaces or trivial (i.e. disintegrated).}

\vspace{0.2cm}

\noindent
The interesting - and remarkable - part of the statement is that the models would be either vector spaces or trivial,
which means  that the reason behind the categoricity of the class is in the realm of classical mathematics. 
Moreover, in some classical settings it can be quite useful to know that a geometry is trivial since it can e.g. rule out algebraic relations between objects.
Thus, trivial in our context is far from uninteresting.

 
However, the statement cannot be proved without something like the joint embedding property (JEP),
as the following counterexample demonstrates.
 
\begin{example}
Let $\o\le\a <\l^{+}$ and let the language
$L$ consist of constants $c_{i}$, $i<\l$, unary relation symbols
$P_{i}$, $i<\a$, a binary relation symbol $R$ and a binary function symbol
$F$. Let $\K_{0}$ be any universal class
of $L$-structures categorical in every cardinality
$\xi >\l=LS(\K_{0} )$ such that $c_{0}=c_{1}$
is true in every model of $\K_{0}$.
Let $\K_{1}$ consist of those  $L$-structures $\A$
such that

\begin{enumerate}[(i)]
\item $\neg c_{i}=c_{j}$ holds for every  $i<j<\l$;
\item $P_{0}$ is the set of all interpretations of the constants $c_{i}$;
\item the sets $P_{i}$, $i<\a$, form a partition of the universe of
$\A$;
\item if $R(a,b)$ holds and $b\in P_{i}$, then $a\in\cup_{j<i}P_{j}$;
\item if there is $0<i<\a$ such that $a,b\in P_{i}$ and $a\ne b$,
then $R(F(a,b),a)\leftrightarrow\neg R(F(a,b),b)$ holds in $\A$.
\end{enumerate}

\noindent
In other words, 
when $i>0$, we attach distinct subsets of $\bigcup_{j<i} P_j$
to distinct elements of $P_i$.
The relation $R$ describes these subsets, 
and $F$ is essentially a Skolem function witnessing that two distinct elements are connected to different sets.
It is easy to see that $(\K_{1},\subseteq )$ is a universal
class with $LS(\K_{1})=\l$ and 
$\K_{1}\cap \K_{0}=\emptyset$.
Moreover,  $\K_{1}$ has a model of power $\xi$ if and only if $\l\le\xi\le\beth_{\a}(\l )$,
where $\beth_{0}(\l )=\l$, $\beth_{\g +1}(\l )=2^{\beth_{\a}(\l )}$
and for limit $\g$, $\beth_{\g}(\l )=\cup_{\delta <\g}\beth_{\delta}(\l )$.
Indeed, the set $P_0$ always has size $\l$,
the set $P_1$ has size at most $2^\l$, and so on.
Now $\K=\K_{0}\cup\K_{1}$ is a universal class
with $LS(\K)=\l$ and
it is categorical in $\xi$ if and only if $\xi >\beth_{\a}(\l )$. 
Obviously it contains models that can not be seen as vector spaces.
\end{example}

However, this example is artificial in the sense that the class $\K$
consists of the class $\K_0$
together with some added noise in the form of the models in $\K_1$.
We have $\K_{>\beth_{\a}(\l )} \subseteq \K_0$,
and above the cardinality $\beth_{\a}(\l )$
the class $\K$ is categorical because $\K_0$ is. 
Even when we tried harder, we didn't manage to find a counterexample that would not be essentially the same as the one above.
This led us to think that maybe Kirby was right,
and we ended up proving that the statement basically holds after removing some relatively small models (noise) from the class.
 
We assume that the class $\K$ satisfies Kirby's assumptions:
$\K$ is a universal class, $LS(\K )=\l$, $\K$ has arbitrarily large models, and $\K$ is categorical in some $\k >\l$.
Moreover, we assume that $\k$ is a regular cardinal such that $\k>\l^+$.
We define $\K^*$ to be the class obtained by removing the noise from $\K$.
More precisely, we take $\K^*$ to consist of all models $\A\in\K_{>\l}$
such that there is some $\B\in\K_{\ge\k}$
for which $\A \subseteq \B$.
We prove that $\K^*$ is totally categorical (Theorem \ref{catbonus}).
Moreover,  $\K_{\ge\beth_{(2^{\l^{+}})^{+}}}\subseteq\K^{*}$ 
(Theorem \ref{jatko}) and the models in $\K^{*}_{>\l^{+}}$ are either essentially vector spaces or trivial (Theorems \ref{quasiurb} and \ref{koordinaatit}).
On the way, we also show that $\K^*$ is an AEC with $LS(\K^*)=\l^+$ (Lemma \ref{aec}).

The gist of our argument is related to geometric stability theory:
we will find a model $\A \in \K_{\l^+}^*$
and a quantifier free type $p \in S_{qf}(\A)$ that is minimal (see Definition \ref{minimal})
and show that if $\B \in \K^*$ contains $\A$,
then the set $X$ of realisations of $p$ in $\B$ has a natural pregeometry.
We will then prove categoricity in $\mu>\l^+$ by showing that 
$\B=\langle \A, a_i \rangle_{i<\a}$, where $(a_i)_{i<\a}$
is a maximal independent sequence of realisations of $p$ (Lemma \ref{generated} and Theorem \ref{categoricity}).
Finally, we will use a result of Zilber's to show that either the pregeometry on $X$ is trivial or $X$  
can be given the structure of a vector space (Theorem \ref{quasiurb}),
and then go on to prove that there is a strong coordinatisation of $\B$
using the elements of $X$ (Theorem \ref{koordinaatit}). 
Together, these two theorems show that the models in $\K^{*}_{>\l^{+}}$ are essentially vector spaces or trivial.
Applying the coordinatisation, 
we show that $\K^*$ is in fact categorical also in $\l^+$ and thus totally categorical.
To sum up, we obtain our results by using geometric stability theory in a non-elementary context. 

In section \ref{cat}, we will prove that the class $\K^*$ is categorical in $\mu>\l^+$,
and on the way, we will show that $\K^*$ is an AEC with $LS(\K^*)=\l^+$.
Much of our proofs rely on the properties of quantifier free indiscernible sequences,
and we need the assumption about arbitrarily large models to build such sequences with Ehrenfeucht-Mostowski constructions. 
As soon as we have constructed a quantifier free indiscernible sequence of size $\k$,
the assumption of $\k$-categoricity can be reformulated by stating that all models of cardinality $\k$
are generated by a quantifier free indiscernible sequence of cardinality $\k$.
Then, we will also have a version of stability (Lemma \ref{stab}),
and using the properties of indiscernibles we can define the average type of an indiscernible sequence (see Definition \ref{avdef}).
We will then use average types to prove the amalgamation property (AP) for the class $\K_{<\k}^*$ (Lemma \ref{AP}).
Given stability, 
we can now use amalgamation to do the usual tree construction to show that there is a minimal type $p$ over some model $\A \in \K_{\l^+}$ (see Definition \ref{minimal} and Lemma \ref{uniqext}). 
In section \ref{vector}, we prove the main result: the structures in $\K^*$ are either trivial or essentially vector spaces
(Theorems \ref{quasiurb} and \ref{koordinaatit}),
and apply it to show that $\K^*$ is totally categorical (Theorem \ref{catbonus}).
On the way to proving these theorems, we show that $\K^*$
has the amalgamation property (AP) and that models in $\K^*$ are saturated (Lemma \ref{apsat}).

 
When using standard arguments from stability theory, 
we omit the details and refer the reader to e.g. \cite{ba} for them.
Most of the stability theoretic techniques we use, like average types,
are originally from \cite{sh}.
The use of geometry in model theory was initiated in \cite{bl},
and the idea behind our proof of upward categoricity transfer comes from this paper.
 
\section{Categoricity}\label{cat}

Let $(\K, \subseteq)$ be a universal class of $L$-structures with arbitrarily large models, $LS(\K)=\lambda= \vert L\vert +\o$, suppose $\k$ is regular, $\k>\l^{+}$,  and $\K$ is $\k$-categorical. 

Denote by $\K^*$ the class consisting of all the models $\A \in \K$ such that $\vert \A \vert > \l$ 
and there is some $\B \in \K_{\ge \k}$ such that $\A \subseteq \B$.
 
If $A$ is a set, we will follow the usual convention in model theory and use the notation $a \in A$
to denote that $a$ is a finite tuple of elements from $A$.
If $a$ is assumed to be a singleton, it will be specifically mentioned.
For a set $A$, we use the notation $\langle A \rangle$ for the model generated by $A$.
 
If $\A \in \K^*$ (or $\A \in \K$) 
and $A \subseteq \A$,
we denote by $\t(a/A)$ the quantifier free $L$-type of $a$ over $A$,
and whenever we talk about types, we mean quantifier free types.
We say that a type $p$ over a set $A$ is \emph{consistent} in the class $\K^*$ (or the class $\K$)
if there is some 
$\B \in \K^*$ (or $\B \in \K$) such that $\langle A \rangle \subseteq \B$ and some $a \in \B$ such that $a$ realises $p$.

\begin{lemma}\label{EM}
There are $\A \in \K_\k$ and $(a_i)_{i<\k} \subseteq \A$ quantifier free order indiscernible over $\emptyset$.
\end{lemma}
 
\begin{proof}
Starting from a model of size at least $\beth_{(2^\l)^+}$, we can do the Ehrenfeucht-Mostowski construction for AECs originally due to Shelah. Since we are working in a universal class, 
we don not need to add Skolem functions, merely apply the Erd\"os-Rado theorem. 
\end{proof}

\begin{lemma}\label{indgen}
Let $\A \in \K_\k$.
Then, there is a quantifier free order indiscernible sequence $(a_i)_{i<\k} \subseteq \A$ such that $\A= \langle a_i \rangle_{i<\k}$. 
\end{lemma}
\begin{proof}
Follows from $\k$-categoricity and Lemma \ref{EM}.
\end{proof}
  
\begin{proof}
By Lemma \ref{indgen} any model of size $\k$ is generated by a quantifier free order indiscernible sequence.
There clearly is a submodel that is as wanted. 
\end{proof}

\begin{lemma}\label{stab}
If $A \subseteq \A \in \K_\k$ and $\vert A \vert < \k$, then $\vert \{\t(a/A) \, | \, a \in \A\} \vert \le \vert A \vert+ \l$.
\end{lemma}

\begin{proof}
By Lemma \ref{indgen}, $\A$ is generated by a quantifier free order indiscernible sequence
$(a_i)_{i<\k}$. 
Now there is some $X \subset \k$ such that $\v X \v \le \v A \v +\o$ and $A \subseteq \langle a_i \rangle_{i \in X}$.
For every $a \in \A$, there is a number $n \in \o$, a $L$-term $t$ and indices $i_1, \ldots, i_n \in \k$ such that
$a=t(a_{i_1}, \ldots, a_{i_n})$.
The type $\t(a/A)$ only depends on the term $t$ and the type $\t(a_{i_1}, \ldots, a_{i_n}/ \langle a_i \rangle_{i \in X})$, and since $(X, <)$ is a well-ordering, there are at most $\v X \v$ many such types.
Since there are $\lambda$ many $L$-terms and $\vert X \vert \le \vert A \vert+\o$, the claim follows.
\end{proof}
 
\begin{lemma}\label{ordind}
If $(a_i)_{i<\a} \subseteq \A \in \K^*$ is an infinite quantifier free order indiscernible sequence over some $A \subseteq \A$ such that $\vert A \vert < \k$, 
then it is quantifier free indiscernible over $A$.
\end{lemma}
 
\begin{proof}
With an Ehrenfeucht-Mostowski construction we can obtain a model that contains an order indiscernible sequence of any order type whose finite subsequences have the same quantifier free type as those of the sequence $(a_i)_{i<\a}$.
Then, the usual argument (see e.g. \cite{Marker}, Theorem 5.2.13 for the first order case) yields a contradiction with Lemma \ref{stab}. 
\end{proof}

\begin{lemma}\label{indgen2}
Let $\A \in \K_\k$.
Then, there is a quantifier free indiscernible sequence $(a_i)_{i<\k} \subseteq \A$ such that $\A= \langle a_i \rangle_{i<\k}$. 
\end{lemma}

\begin{proof}
Follows from Lemmas \ref{indgen} and \ref{ordind}.
\end{proof}

\begin{lemma}\label{delta0}
Suppose $\A  \in \K_{\le \k}^*$, $\mu>\l$ is a regular cardinal, $(a_i)_{i<\mu}$ is a sequence of distinct tuples from $\A$,
and $A \subseteq \A$ is such that
$\v A \v < \mu$.
Then, the sequence $(a_i)_{i<\mu}$ contains a subsequence of length $\mu$ that is quantifier free indiscernible over $A$. \end{lemma}
 
\begin{proof}
For simplicity of notation, assume the $a_i$ are singletons.
By the definition of the class $\K^*$, there is some model $\B \in \K_\k^*$ such that 
$\A \subseteq \B$, and by Lemma \ref{indgen2}, the model $\B$ 
is generated by some sequence $(b_i)_{i<\k}$ indiscernible over $\emptyset$.
 
For each $j<\mu$, we can write $a_j=t_j(b_{i_0}, \ldots, b_{i_n})$, where $t_j$ is an $L$-term and $i_0, \ldots, i_n < \k$.
Since $\mu>\lambda$, we may without loss assume (use the pigeonhole principle) that there is an $L$-term $t$ such that $t_j=t$ for all $j<\mu$.
Then, for each $j<\mu$, we can write $a_j=t(b_{i_0}^j, \ldots, b_{i_n}^j)$.

There is some $X \subset \k$ such that $\v X \v < \mu$ and $A \subseteq \langle b_k \rangle_{k \in X}$, 
and since $\mu$ is regular, we can use the $\Delta$-lemma and thus we may assume that there is a finite set $Y$ such that 
$\{i_0^j, \ldots, i_n^j\} \cap \{i_0^k, \ldots, i_n^k\}=Y$ for all $j<k<\mu$.
Using the pigeonhole principle, we may assume that if $i_p^j \in Y$, then $i_p^j=i_p^k$ for all $k<\mu$.
Since $\vert X \vert < \mu$, we may assume (using the pigeonhole principle) that $i_p^j \notin Y$ implies $i_p^j \notin X$ for all $j < \mu$.

Now it is easy to see that we have obtained a sequence of length $\mu$ that is quantifier free indiscernible over $A$.
\end{proof}

\begin{lemma}\label{delta}
Suppose $A \subseteq \A \in \K_{\le \k}^*$ and $\mu$ is a regular cardinal such that $\l<\mu$ and
$\vert A \vert < \mu \le \vert \A \vert$.
Then there is a sequence $(a_i)_{i<\mu}$ quantifier free indiscernible over $A$.
\end{lemma}

\begin{proof}
Choose any non-trivial sequence $(a_j)_{j<\mu} \subseteq \A$ and apply Lemma \ref{delta0}.
\end{proof}

\begin{definition}\label{satdef}
A model $\A \in \K^*$ is \emph{saturated} if for all models $\B, \C \in \K^*$ such that $\C \subseteq \A, \B$ and $\v \C \v < \v \A \v$, and any $b \in \B$, the quantifier free type $t_{qf}(b/\C)$ is realised in $\A$.

A model $\A \in \K^*$ is \emph{weakly saturated} if for all models $\B \in \K^*$ and all sets $A$ such that $A \subseteq \A \subseteq \B$ and $\vert A \vert < \vert \A \vert$, and any $b \in \B$, the quantifier free type $t_{qf}(b/A)$ is realised in $\A$. 
\end{definition}

\begin{lemma}\label{weaksat}
Every model in $\K_\k$ is weakly saturated.
\end{lemma}

\begin{proof}
We will construct a weakly saturated model of size $\kappa$, and the result will then follow from $\k$-categoricity.

   
We will build models $\A_i \in \K_\k$, $i \le\k$, and show that $\A_\k$ is weakly saturated.
Let $\pi: \k \to (\k \cup \{-1\}) \times \k$ be a bijection such that if $i<\k$ and $\pi(i)=(j,k)$, then $j<i$.
For each $i\le \k$, denote $\pi[i]=\{\pi(j) \, | \, j<i \}$.

\begin{claim}\label{claim1}
There are models $\A_j \in \K$, $j  \le \k$ such that $dom(\A_j)=\{(\a, \b) \in \pi(\k) \, | \, \a<j, \b <\k\}$,
$\A_j \subseteq \A_k$ whenever $j<k$, and for each $j$, the model $\A_{j+1}$ satisfies a maximal collection of quantifier free types over the model 
$\langle \pi[j] \rangle$.
In other words, for each $j<\k$, there is a set $X \subseteq \v j \v +\l$ and a collection of quantifier free types $(p_i)_{i \in X}$ over the model $\langle \pi[j] \rangle$ satisfying the following conditions:
\begin{itemize} 
\item $\A_{j+1}$ realises $p_i$ for each $i \in X$;
\item $X$ is a maximal such collection: if there is some $\C \in \K^*$ such that $\A_{j+1} \subseteq \C$ and $\C$ realises some quantifier free type $p$ over $\langle \pi[j] \rangle$, then $p=p_i$ for some $i \in X$.
\end{itemize}
\end{claim}
 
\begin{proof}
To construct $\A_0$, just take any model from $\K_\k$ with the right universe.
For a limit ordinal $i$, let $\A_i=\bigcup_{j<i} \A_j$.
For each $j$, we construct the model $\A_{j+1}$ and the set $X$ as follows.

First, we start constructing an increasing chain of models $\B_i$ and sets $X_i$ in the following way
(this process will eventually terminate in less that $(\v j \v + \l)^+$ steps, as we shall see). 
Set $\B_0=\A_j$ and $X_0=\emptyset$. 
For each $i$, construct $\B_{i+1}$ and $X_{i+1}$ as follows (at limit steps, take unions).
If there is some $\C \in \K^*$ such that $\B_i \subseteq \C$ and some $b \in \C \setminus \B_i$
such that $\t(b/\langle \pi[j] \rangle)$ is not satisfied in $\B_i$, let $\B_{i+1}=\langle \B_i, b \rangle$,
denote $p_i=\t(b/\langle \pi[j] \rangle)$ and let $X_{i+1}=X_i \cup \{i \}$;
otherwise set $\A_{j+1}=\B_{i}$ and $X=X_{i}.$  

The process terminates in less than $(\v j \v + \l)^+$  steps.
Indeed, if not, we would have constructed a model $\B^+$ that satisfies $(\v j \v +\l)^+$ 
many types over $\langle \pi[j] \rangle$, a set of size $\v j \v$, a contradiction against Lemma \ref{stab}.
Thus, we have added at most $\k$ many elements and obtained a model $\A_j$ that has a universe of size $\k$.
Rename the new elements (and add some more elements if needed) so that they form the set $\{(j, \b)  \, | \, \b<\k \}$.
 \end{proof} 
  
 We claim that the model $\A_\k$ built this way is weakly saturated. 
 Suppose $A \subseteq \A_\k$, $\v a \v < \k$, let $\C \supseteq \A_\k$ and let $b \in \C$.
 Now, there is some $i<\k$ such that $A \subseteq \pi[i]$.
 By Claim \ref{claim1}, there is some $b' \in \A_{i+1}$ such that $\t(b'/ \langle \pi[i] \rangle)=t(b/\langle \pi[i] \rangle)$.
 Since $\A_{i+1} \subseteq \A_\k$, this proves that $\A_\k$ is weakly saturated.
\end{proof}

\begin{lemma}\label{pienityyppi}
Suppose $A \subseteq \B$, $\B \in \K_{<\k}^*$,   the sequence $(a_i)_{i<\l^+} \subseteq \B$ is quantifier free indiscernible over $A$, $\v A \v < \l^+$, let $\C \in \K^*$ be such that $\B \subseteq \C$, and let
$b \in \C$.
Then, there is a quantifier free type $p$ such that 
$$\vert \{i<\a \, | \, \t(a_ib/A) \neq p\}| < \l^+.$$ 
\end{lemma}

\begin{proof}
If not, then there is a quantifier free type $p$ such that 
$$\vert \{i<\l^+ \, | \, \t(a_ib/A) = p\}| =\vert \{i<\l^+ \, | \, \t(a_ib/A) \neq p\}|=\l^+.$$
We may without loss assume $\t(a_ib/A)=p$ if and only if $i$ is even (i.e. $i=2j$ for some ordinal $j$).
Since $\C \in \K^*$, we may apply  Lemma \ref{delta0} to find some
$X \subseteq \l^+$ such that $\v X \v = \l^+$ and the sequence of ordered pairs
$(a_{2i}, a_{2i+1})_{i \in X}$ is quantifier free indiscernible over $Ab$.

Relabeling the elements of $X$, we may assume $X=\l^+$, 
and with an Ehrenfeucht-Mostowski construction, we find a model $\mathcal{M} \in \K^*$ containing $Ab$ and some $a_i \in \mathcal{M}$, $\l^+\le i <\k$,
such that $(a_{2i}, a_{2i+1})_{i < \k}$ is quantifier free indiscernible over $Ab$.  
Note that the sequence $(a_i)_{i<\k}$ obtained this way is quantifier free indiscernible over $A$.
Let $\D=\langle A,b,a_i \rangle_{i<\k}$ and $\D'=\langle A,a_i \rangle_{i<\k}$.
By Lemma \ref{weaksat}, there is some $b' \in \D'$ such that $\t(b'/\langle A a_i \rangle_{i<\l^+})=\t(b/\langle A a_i \rangle_{i<\l^+})$.

For each of the $2^{\l^+}$ many sets $Y \subset \l^+$ such that $\v Y \v =\v \l^+ \setminus Y \v =\l^+$,
there is a permutation of the sequence $(a_i)_{i<\l^+}$ 
that takes even indices to $Y$ and odd indices to $\l^+ \setminus Y$.
These extend to permutations of the sequence $(a_i)_{i<\k}$ 
(just fix $a_i$ for each $i \ge \l^+$) and thus to automorphisms $f_i$, $i<2^{\l^+}$, of $\D'$. 
If $i \neq j$, then $\t(f_i(b')/ A \cup \{a_i\}_{i<\l^+}) \neq \t(f_j(b')/ A \cup \{a_i\}_{i<\l^+})$,
so we get $2^{\l^+}$ many distinct quantifier free types over a set of size $\l^{+}$, which contradicts Lemma \ref{stab}.
 


 
\end{proof}

Lemma \ref{pienityyppi} makes it possible to define the notion of average type for a quantifier free indiscernible sequence.

\begin{definition}\label{avdef}
Let $\a \ge \l^+$, let $\A \in \K^*$, let $(a_i)_{i<\a} \subseteq \A$ be a quantifier free indiscernible sequence, and let $A \subseteq \A$.
We say a tuple $b$ satisfies the \emph{average type} over $A$  of the sequence (denoted $Av((a_i)_{i<\a}/A))$ if it holds for any $a \in A$ that $\v\{i<\a \, | \, \t(a_ia/\emptyset)=\t(ba/\emptyset)\} \v=\v \a \v$.
\end{definition}

Note that eventhough it might naively seem so,
Lemma \ref{pienityyppi} does not actually put any restrictions on the size of the set $A$ in Definition \ref{avdef}.
In the statement of the lemma, there is a set denoted by $A$ that needs to be smaller than $\l^+$, 
but in the definition, this set is assumed to be empty, so the requirement is satisfied automatically.
In contrast, the set denoted by $A$ in the definition has a different role:
the only requirement it is assumed to satisfy as a set is being included in $\A$, 
and the main condition is set only for finite tuples from $A$, which here play the role that the tuple $b$ 
has in the statement of Lemma \ref{pienityyppi}.
 
We also note that precisely since it is the finite tuples from $A$ 
that are considered when checking the requirements of Definition \ref{avdef},
the notion of an average type can also be formulated in the more general case where
$\A \in \K$, as long as there is some submodel $\B \subseteq \A$ such that $\B \in \K^*$,
$(a_i)_{i<\a} \subseteq \B$,
and for each finite tuple $a \in \A$, it holds that $\langle \B, a \rangle \in \K^*$. 
To prove that the class $\K^*$ is an AEC, 
we will, at one point, need the notion in this broader context.
Thus, we will also give the following definition of $\K$ average types as a generalisation of Definition \ref{avdef}.
Although Definition \ref{avdef} is a special case of this more general definition,
we will use it as our main definition for average types since it is more intuitive,
and the more general definition is needed only once in this paper.

\begin{definition}\label{avdefK}
Let $\a \ge \l^+$, let $\A \in \K$,  and let $I=(a_i)_{i<\a} \subseteq \A$ be a quantifier free indiscernible sequence.
Suppose there is some $\B \subseteq \A$ such that $\B \in \K^*$, $I \subseteq \B$, and $\langle \B, a \rangle \in \K^*$
for each $a \in \A$.
We say a tuple $b$ satisfies the $\K$ \emph{average type} over $\A$ of the sequence (denoted $Av_\K(I/\A))$ if it holds for any $a \in \A$ that $\v\{i<\a \, | \, \t(a_ia/\emptyset)=\t(ba/\emptyset)\} \v=\v \a \v$.
\end{definition}

\begin{lemma}\label{avexists0}
Suppose $I=(a_i)_{i<\a} \subseteq \A \in \K$ is quantifier free indiscernible over  $\emptyset$, $\a \ge \l^+$, 
and there is some $\B \subseteq \A$ such that $\B \in \K^*$, $I \subseteq \B$, and $\langle \B, a \rangle \in \K^*$
for each $a \in \A$. 
Then, there is some $\C \in \K$ such that $\A \subseteq \C$ and $b \in \C$ such that $\t(b/\A)=Av_\K(I/\A)$. 
\end{lemma}
 
\begin{proof}
By compactness of first order logic,
there is a model $\D$ (not necessarily in $\K$)
such that $\A \subseteq \D$ and some $b \in \D$
such that $\t(b/\A)=Av(I/\A)$.
Now, for each $a \in \A$, there is some $i<\a$
such that $\langle a, b \rangle \cong \langle a, a_i \rangle$,
so $\langle a, b \rangle \in \K$.
Thus, $\C=\langle \A, b \rangle$
is a union of a directed system of models from $\K$,
so $\C \in \K$.

 
 \end{proof} 
  
 \begin{corollary}\label{avexists}
 Suppose $I=(a_i)_{i<\a} \subseteq \A \in \K^*$ is quantifier free indiscernible over  $\emptyset$  and $\a \ge \l^+$. 
Then, there is some $\B \in \K^*$ such that $\A \subseteq \B$ and $b \in \B$ such that $\t(b/\A)=Av(I/\A)$.

Furthermore, $I^\frown b$ is indiscernible over $A$ if and only if $\t(b/A \cup I)=Av(I/A \cup I)$.
 \end{corollary}
 
\begin{proof}
Since $\A \in \K^*$, there is some $\A' \in \K_{\ge\k}$ such that $\A \subseteq \A'$.
By Lemma \ref{avexists0}, there is some $\B \in \K$ and $b \in \B$ such that
$\t(b/\A')=Av_\K(I/\A')=Av(I/\A')$.
Since $\B$ contains $\A'$ and $\v \A' \v \ge \k$, we have $\B \in \K^*$.
 
The furthermore part is straightforward application of the definitions, as usually.
\end{proof}  
 
The next lemma is the key for showing that  $\K^*$ actually is an AEC.
Its proof is the only point where we need the more general definition of average types (Definition \ref{avdefK}). 
 
\begin{lemma}\label{peruslemma}
The following are equivalent:
\begin{enumerate}[(i)]
\item $\A \in \K^*$;
\item If $\B \subseteq \A$ and $\v \B \v=\l^+$, then $\B \in \K^*$;
\item There is some $\B \subseteq \A$ such that $\v \B \v=\l^+$ and for all  $a \in \A$, $\langle \B, a \rangle \in \K^*$.
\end{enumerate}
\end{lemma}
 
\begin{proof}
Clearly (i) $\Rightarrow$ (ii) $\Rightarrow$ (iii).

For (iii) $\Rightarrow$ (i), we note that by Lemma \ref{delta}, there is some quantifier indiscernible sequence $I=(a_i)_{i<\l^+} \subseteq \B$.
For each $a \in \A$, we have $\langle \B, a \rangle \in \K^*$, so the type $Av_\K(I/\A)$ is well-defined.
By Lemma \ref{avexists0}, there is some $b$ that realises $Av_\K(I/\A)$ and $\A'=\langle \A, b\rangle \in \K$.

 
We claim that $\A'$ satisfies (iii).
For this, it suffices to show that for any tuple $a \in \A$, it holds that $\langle \B, a, b \rangle \in \K^*$.
By our assumptions, there is some $\C \in \K_\k$ such that $\langle \B, a \rangle \subseteq \C$.  
The type $Av(I/\C)$ is consistent so by Lemma \ref{weaksat} 
we may without loss assume $b \in \C$.
Hence, $\A'$ satisfies (iii).

Any union of a chain of models that satisfy (iii) will also satisfy (iii), so setting $\A_0=\A$, $\A_1=\A'$, and so on, 
we can continue the construction until we obtain a model of size $\k$. 
\end{proof}

\begin{lemma}\label{aec}
The class $\K^*$ is an AEC, and $LS(\K^*)=\l^+$.
\end{lemma}

\begin{proof}
It follows from Lemma \ref{peruslemma} that unions of chains of models from $\K^*$ are in $\K^*$.
The other requirements are easy to check. 
\end{proof}

We will assume that all models mentioned from now on are in the class $\K^*$ unless otherwise mentioned.
 
\begin{remark}\label{avremark}
Note that if $\A \in \K^*$ and $I \subset \A$ is an indiscernible sequence of length at least $\l^+$, then Corollary \ref{avexists} makes it possible to find some $b$ such that $b$ realises $Av(I/\A)$ and $\langle \A, b \rangle \in \K^*$.
Indeed, the lemma states that there is a model $\B \in \K^*$ and $b \in \B$ that is as wanted.
Since $\A, \B \in \K^*$ and $\A \subseteq \langle \A, b \rangle \subseteq \B$, we have $\langle \A, b \rangle \in \K^*$.
Thus, the formulation of the next lemma makes sense.
\end{remark}

\begin{lemma}\label{avind}
Suppose $\A \in \K^*$, $I=(a_i)_{i<\a} \subseteq \A$ is indiscernible over $A \subseteq \A$, $\v A \v < \a$, and
$\a \ge \l^+$.
Let $\b > \v A \v$, and let $\B_i=\langle \A, b_j \rangle_{j<i}$, $i<\b$, be such that $\t(b_i/\B_i)=Av(I/\B_i)$.
Then, the sequence $(b_i)_{i<\b}$ is quantifier free indiscernible over $\A$.
\end{lemma}

\begin{proof}
By lemma \ref{ordind}, it suffices to show that the sequence is order indiscernible, i.e. that for all $i_1 < \ldots < i_n <\b$ and $j_1 < \ldots < j_n < \b$, it holds that $\t(b_{i_1}, \ldots, b_{i_n}/\A)=\t(b_{j_1}, \ldots, b_{j_n}/\A)$.

If $i<j<\b$ and $a \in \A$, then by the definition of average types, there is a set $X \subseteq \a$ such that $\v X \v =\a$,
$\v \a \setminus X \v < \a$ and for each $k \in X$,
$$\t(b_i a/\emptyset)=\t(a_k a/\emptyset)=\t/(b_ja/\emptyset),$$
so $\t(b_i/\A)=\t(b_j/\A)$.

Suppose that for all $a \in \A$, $i_1 < \ldots < i_n <\b$ and $j_1 < \ldots < j_n <\b$, it holds that
$\t(b_{i_1}, \ldots, b_{i_n}/\A)=\t(b_{j_1}, \ldots, b_{j_n} /\A)$.
For the induction step, let $a \in \A$, $i_1 < \ldots < i_{n+1} < \b$ and $bj_1 < \ldots < j_{n+1} < \b$.
By our assumptions and the definition of average types, there is a set $Y \subseteq \a$ such that $\v Y \v =\a$,
$\v \a \setminus Y \v < \a$, and for each $k \in Y$,
$$\t(b_{i_1}, \ldots, b_{i_n}, b_{i_{n+1}}, a / \emptyset)=\t(b_{i_1}, \ldots, b_{i_n}, a_k, a / \emptyset),$$
and 
$$\t(b_{j_1}, \ldots, b_{j_n}, b_{j_{n+1}}, a / \emptyset)=\t(b_{j_1}, \ldots, b_{j_n}, a_k, a / \emptyset).$$
Since $a_k \in \A$,
the inductive assumption gives us
$$\t(b_{i_1}, \ldots, b_{i_n}, a_k, a / \emptyset)=\t(b_{j_1}, \ldots, b_{j_n}, a_k, a / \emptyset),$$
and thus 
$$\t(b_{i_1}, \ldots, b_{i_n}, b_{i_{n+1}}, a / \emptyset)=\t(b_{j_1}, \ldots, b_{j_n}, b_{j_{n+1}}, a / \emptyset),$$
as wanted.
\end{proof}

\begin{lemma}\label{universal}
If $\A \in \K_{\le \k}^*$ and $\mu$ is a regular cardinal such that $\l<\mu \le \vert \A \vert$, then $\A$ is $\mu$-universal.
\end{lemma}

\begin{proof}
 By Lemma \ref{delta}, there is a sequence $(a_i)_{i<\mu} \subseteq \A$ quantifier free indiscernible over $\emptyset$.
It suffices to show that the model generated by this sequence is $\mu$-universal, and thus we may without loss assume that $\A=\langle (a_i)_{i<\mu} \rangle$.

By an Ehrenfeucht-Mostowski construction, we obtain a model $\B \in \K^*_\k$ such that $\A \subseteq \B$ and a sequence $(a_i)_{\mu \le i < \k} \subseteq \B$ such that $\B=\langle (a_i)_{i<\k} \rangle$.
By $\k$-categoricity and the assumptions on the class $\K^*$, the model $\B$ is $\k$-universal.
Thus, $\A$ is $\mu$-universal (just use a suitable permutation of $(a_i)_{i<\k}$). 
\end{proof}

\begin{lemma}\label{apu}
Let $\A \in \K^*_{< \k}$, $\B \supseteq \A$, $b \in \B$, let $A \subset \A$, and suppose  
$I=(a_i)_{i<\l^{+}} \subseteq \A$ is a sequence quantifier free indiscernible over $A$.
Then, $\A$ realises $\t(b/A)$.
\end{lemma}
 
\begin{proof}
We show that $\t(b/A)$ is realised already in $\langle A \cup I \rangle$.
Choose tuples $a_i$, for $\l^+ \le i < \k$, so that 
$$\t(a_i/ \langle \B, a_j \rangle_{j<i})=Av(I/\langle \B, a_j \rangle_{j<i}).$$
Denote $\D=\langle A, a_i \rangle_{i<\k}$
and $\D'= \langle \B, a_i \rangle_{i<\k}$.
Note that $\langle \A, a_i \rangle_{i<\k} \in \K^*$, and thus $\D \in \K^*$.
Since $b \in \D' \supseteq \D$, the model $\D$ realises $\t(b/A)$
by Lemma \ref{weaksat}.
Thus, there is some $L$-term $t$, some tuple $a \in A$ and some $i_0 < \ldots < i_n < \k$ such that $b=t(a_{i_0}, \ldots, a_{i_n}, a)$.
By Corollary \ref{avexists}, the sequence $(a_i)_{i<\k}$ is quantifier free indiscernible over $A$.
Thus, if $j_0< \ldots <j_n<\l^+$, then also $c=t(a_{j_0}, \ldots, a_{j_n})$ realises $\t(b/A)$. Clearly $c \in \langle A \cup I \rangle$.
\end{proof} 
 
\begin{lemma}\label{weaksat2}
If $\A \in \K_{\le \k}^*$, then $\A$ is weakly saturated.
\end{lemma}
 
 \begin{proof}
Let $A \subseteq \A$ and suppose $\v A \v < \v \A \v$, and let $b \in \C \supseteq \A$.
If $\v \A \v$ is a limit cardinal, there is some successor cardinal $\mu^+< \v \A \v$ such that $\v A \v<\mu^+$,
so we may without loss assume $\v \A \v=\mu^+$.
Considering how the class $\K^*$ is defined, we have $\mu^+ \ge \l^+$.
By Lemma \ref{delta}, there is a sequence $I=(a_i)_{i<\mu^+} \subseteq \A$ quantifier free indiscernible over $A$. 
By Lemma \ref{apu}, $\A$ realises $\t(b/A)$.
\end{proof}
  
 \begin{lemma}\label{AP}
 The class $\K^*_{<\k}$ has the amalgamation property (AP).
 \end{lemma}

\begin{proof}
Suppose $\A, \B, \C \in \K^*_{<\k}$, $\A \subseteq \B, \C$, and let $b \in \C \setminus \A$.
It suffices to find a model $\B^* \supseteq \B$ and an embedding $f: \langle \A, b \rangle \to \B^*$ such that $f \raj \A=id$.
By Lemma \ref{delta}, there is a sequence $I=(a_i)_{i<\l^+} \subseteq \A$, quantifier free indiscernible over $\emptyset$.

By Corollary \ref{avexists}, we may choose some $b_0$ and $c_0$ so that $\t(b_0/\B)=Av(I/\B)$ and $\t(c_0/\C)=Av(I/\C)$.
Since $\t(b_0/\A)=Av(I/\A)=\t(c_0/\A)$, we have
$\langle \A, b_0 \rangle \cong \langle \A, c_0 \rangle$, and thus we may assume (moving $\B$ if necessary) $b_0=c_0$ and $\langle \B, b_0 \rangle \cap \langle \C, b_0 \rangle = \langle \A, b_0 \rangle$.

By Remark \ref{avremark}, there are, for $i<\l^{+}$, $\B_i$ and $b_i$ such that $\t(b_i/ \B_i)=Av(I/\B_i)$ and $\B_i=\langle \B, b_j \rangle_{j<i}$.
Construct $c_i$ and $\C_i$ for $i<\l^+$ similarly.
Again, by moving $\B_i$ at each step if necessary, we may assume that for each $i$, $c_i=b_i$ 
and $\B_i \cap \C_i=\langle \A, b_j \rangle_{j<i}$.
 
Now, $\langle \A, b_i \rangle_{i<\l^+} \subseteq \langle \B, b_i \rangle_{i<\l^+}$, and $b \in \langle \C, b_i \rangle_{i<\l^+}$.
By Lemma \ref{avind}, the sequence $(b_i)_{i<\l^+}$ is indiscernible over $\A$.
Denote $\A'=\langle \A, b_i \rangle_{i<\l^+}$.
Now, $\A' \subseteq \bigcup_{i<\l^+}\C_i$, and thus by Lemma \ref{apu}, there is some $b' \in \A'$ such that $\t(b'/\A)=\t(b/\A)$.
Since $\A' \subseteq \bigcup_{i<\l^+}\B_i$, we may choose $\B^*=\bigcup_{i<\l^+}\B_i$, and we are done. 
\end{proof}

\begin{corollary}\label{saturated}
If $\A \in \K^*_{<\k}$ and $\v \A \v>\l^+$, then $\A$ is saturated.
\end{corollary} 

\begin{proof}
Let $\C \in \K^*$, and let $\B \subseteq \A, \C$ be such that $\v \B \v < \v \A \v$, and let $b \in \C$.
We may without loss assume $\v \C \v <\k$.
By Lemma \ref{AP}, there is some $\D \in \K^*$ such that $\A \subseteq \D$ and an embedding
$f: \C \to \D$ such that $f \upharpoonright \B =id$, 
and thus by Lemma \ref{weaksat2}, there is some $b' \in \A$ such that $\t(b'/\B)=\t(b/\B)$. 
\end{proof}

\begin{corollary}\label{saturated2}
If $\A \in \K_{\le \k}$, then $\A$ is saturated.
\end{corollary}

\begin{proof}
If $\v \A \v < \k$, then the result follows from Corollary \ref{saturated}.
The class $\K$ is $\k$-categorical, so to prove the corollary for models of size $\k$, it suffices to construct a saturated model of size $\k$.
Build a chain of models $\A_i \in \K^*$, $i<\k$, such that $\v \A_i \v < \k$, as follows.
Let $\A_0$ be any model of size $\l^+$.
By Lemmas \ref{stab} and \ref{AP}, there are less than $\k$ many consistent types over the model $\A_i$,
and using Lemma \ref{AP}, we can construct a model $\A_{i+1}$ that satisfies all of them.
At limit steps, take unions.
 
Now, the model $\A=\bigcup_{i<\k} \A_i$ is of size $\k$, and we claim it is saturated.
Indeed, let $\B \subseteq \A$ be such that $\v \B \v < \k$ and let $\C \in \K^*$ be such that $\B \subseteq \C$
and let $a \in \C$.
We need to realise $\t(a/\B)$ in $\A$.
We may without loss assume $\C=\langle \B, a \rangle$ and in particular $\v \C \v < \k$.
There is some $i<\k$ such that $\B \subseteq \A_i$.
By Lemma \ref{AP}, there is some model $\C'$ such that $\A_i \subseteq \C'$ and $\C$ can be embedded into $\C'$.
Since $\C'$ realises $\t(a/\A_i)$, this type (and hence also $\t(a/\B)$) is realised by construction in $\A_{i+1} \subseteq \A$. 
 \end{proof}

\begin{corollary}\label{cate0}
The class $\K^*$ is $\mu$-categorical whenever $\l^+<\mu \le \k$.
\end{corollary}

\begin{proof}
By our assumptions, the class $\K^*$ is $\k$-categorical, so we may assume $\mu<\k$.
Let $\B, \B' \in \K^*_{\mu}$,
and let $\A \subseteq \B$ be such that $\v \A \v= \l^+$.
By Lemma \ref{universal},  
$\B'$ contains an isomorphic copy of $\A$.
We can use Corollary \ref{saturated}
to prove $\B \cong \B'$ with the usual back and forth construction.
\end{proof}

\begin{definition}\label{minimal}
Let $\A \in \K^*$.
We say that a type $p \in S_{qf}(\A)$ is \emph{minimal} if 
\begin{itemize}
\item 
there is some
$\A' \in \K^*$ such that $\A \subseteq \A'$
and some $a \in \A' \setminus \A$ such that $\t(a/\A)=p$; and
\item whenever $\A \subseteq \B \subseteq \C$, $a,b \in \C \setminus \B$ and $\t(a/\A)=\t(b/\A)=p$, then $\t(a/\B)=\t(b/\B)$. 
\end{itemize}
 \end{definition} 
  
\begin{lemma}\label{uniqext}
There is a model $\A \in \K^*_{\l^+}$ and a type $p \in S_{qf}(\A)$ such that $p$ is minimal. 
\end{lemma}

\begin{proof}
By Lemma \ref{AP}, we have amalgamation.
Thus, if the statement of the lemma does not hold, 
 we can do the usual tree construction to get more than $\l^+$ many types over a set of size $\l^+$ which contradicts Lemma \ref{stab}.
\end{proof}

Fix now a model $\A$ and a type $p \in S(\A)$ that are as in the statement of Lemma \ref{uniqext}.

\begin{lemma}\label{models}
There are models $\B_i \in \K^*_{\le \l^{++}}$, for $i<\l^{++}$, such that $\B_0=\A$, $\B_i \subseteq \B_j$ whenever $i<j$,
and some $b_i \in \B_{i+1} \setminus \B_i$ such that $\t(b_i/\A)=p$ and $\B_{i+1}=\langle \B_i, b_i \rangle$.

Furthermore, any infinite sequence of tuples $b_i$ satisfying the above requirements is indiscernible over $\A$.
\end{lemma}

\begin{proof}
By Lemma \ref{delta}, there is a sequence $I=(a_i)_{i<\l^+} \subseteq \A$ quantifier free indiscernible over $\emptyset$.
Let $b$ be a tuple from some extension of $\A$ such that $\t(b/\A)=p$, and let $\C=\langle \A, b \rangle$.
For $i<\l^{+++}$, construct models $\C_i$ and choose tuples $c_i$ so that $\C_0=\C$,  $\t(c_i/\C_i)=Av(I/\C_i)$, $\C_{i+1}=\langle \C_i, c_i \rangle$, and if $i$ is a limit ordinal, then $\C_i=\bigcup_{j<i} \C_j$.
Denote $\D=\langle \A, c_i \rangle_{i<\l^{+++}}$.

We note first that $b \notin \D$.
Indeed, if we had $b \in \D$, then we would have $b=t(a, c_{i_0}, \ldots, c_{i_n})$ for some $L$-term $t$, some tuple $a \in \A$, and some $i_0< \ldots< i_n < \l^{+++}$.
By Lemma \ref{delta0}, there is some $J \subseteq I$ of length $\l^+$ such that $J$ is indiscernible over $ab$.
Since the sequences $J$ and $I$ have the same average types, 
each $c_i$ realises the type $Av(J/ab \cup \{c_j\}_{j<i} \cup J)$,
and thus the sequence $J \cup \{c_i\}_{i<\l^{+++}}$ is indiscernible over $ab$ by Lemma \ref{avind}.
Hence, we can find $j_0, \ldots, j_n \in J$ such that $b=t(a, a_{j_0}, \ldots, a_{j_n}) \in \A$, a contradiction.

We will now construct models $\B_i$ for $i<\l^{++}$ so that for each $i$, $\B_i \subseteq \D$ and $\v \B_i \v < \v \D \v$.
Let $\B_0=\A$. 
If we have defined $\B_i$, there is by lemma \ref{weaksat2}, some $b_i \in \D$ such that $\t(b_i/\B_i)=p$.
Let $\B_{i+1}=\langle \B_i, b_i \rangle$.
If $i$ is a limit ordinal, let $\B_i=\bigcup_{j<i} \B_j$.
\end{proof}

\begin{corollary}\label{exists}
If $\A \subseteq \B \in \K^*$,
then there is some $\C \in \K^*$
such that $\B \subseteq \C$, and there is some $a \in \C \setminus \B$ such that $\t(a/\A)=p$.
\end{corollary}

\begin{proof}
Let the models $\B_i$, $i < \l^{++}$, be as in the statement of Lemma \ref{models},
and denote $\B'=\bigcup_{i<\l^{++}} \B_i$.
Suppose first $\v \B \v \ge \l^{++}$.
If $\v \B \v \le \k$,
then we use Corollary \ref{saturated} to construct an embedding $f: \B' \to \B$ such that $f\raj \A=id$,
and if $\v \B \v > \k$, then we embed $\B'$ similarly to some submodel of $\B^* \subseteq \B$
such that $\A \subseteq \B^*$ and $\v \B^* \v=\k$.
Thus, there is an indiscernible sequence $(b_i)_{i<\l^{++}} \subseteq \B$ such that $\t(b_i/\A)=p$
for each $i < \l^{++}$.
By Corollary \ref{avexists}, there is some $\C \supseteq \B$ and some $a \in \C$
such that $\t(a/\B)=Av(I/\B)$, and now $a$ and $\C$ are as wanted.

Suppose now $\v \B \v =\l^+$, and denote $\B''=\langle \A, b_i \rangle_{i<\l^+}$.
Using Lemma \ref{AP}, we amalgamate the models $\B$ and $\B''$ over $\A$
to obtain a model $\C' \supseteq \B$ and an indiscernible sequence
$J=(b_i')_{i<\l^+} \subseteq \C'$ such that $\t(b_i'/\A)=p$ for each $i<\l^+$.
Using Corollary \ref{avexists}, we find a model $\C \supseteq \C'$
and some $a \in \C$ such that $\t(a/\C')=Av(J/\C')$.
These are as wanted.
\end{proof}

\begin{definition}
Let $\B$ and $\C$ be models and $A$ a set such that $A \subseteq \B \subseteq \C$,
and let $a \in \C$.
We say that the type $\t(a/\B)$ \emph{splits} over $A$ if there are some $b, c \in \B$ such that 
$\t(b/A)=\t(c/A)$ but $\t(ab/A) \neq \t(ac/A)$.
\end{definition}

\begin{lemma}\label{dnsplit}
Suppose $\B, \C \in \K^*_{\le \k}$, $\A \subseteq \B \subseteq \C$, and $a \in \C \setminus \B$ is such that $\t(a/\A)=p$.
Then, $\t(a/\B)$ does not split over $\A$.
\end{lemma}

\begin{proof} 
Assume first $\v \B \v > \l^+$.
Suppose for the sake of contradiction that $\t(a/\B)$ splits over $\A$.
Then, there are some $b, c \in \B$ such that $\t(b/\A)=\t(c/\A)$ but $\t(ab/\A) \neq \t(ac/\A)$.
Now, $\langle \A, b \rangle \cong \langle \A, c \rangle$,
and since $\B$ is saturated by Lemma \ref{saturated2},
there is an automorphism $f$ of $\B$ fixing $\A$ pointwise such that $f(b)=c$. 
By closing the model $\langle \A, b,c \rangle$
with respect to $f$ and its inverse,
we get a model  $\B' \subseteq \B$ such that $\v \B' \v =\l^+$,
$\A \subseteq \B'$, $b,c \in \B'$ and $f'=f \raj \B'$ is an automorphism of $\B'$. 
By Lemma \ref{saturated2}, $\C$ is saturated,
so $f'$ extends to an automorphism $g$ of $\C$ that fixes $\A$.
Now, $g(a)=a' \neq a$, and we get
$\t(ad/\A) \neq \t(ac/\A)=\t(a'd/\A)$, so $\t(a'/\B') \neq \t(a/\B')$ which contradicts the minimality of $p$.

Suppose now $\v \B \v =\l^+$. 
We may without loss assume that $\v \C \v =\l^{++}$.
By Corollary \ref{exists},
there is some $\C' \in \K^*$ such that $\C \subseteq \C'$
and some $a' \in \C' \setminus \C$ such that $\t(a'/\A)=p$.
By what we proved above, $\t(a'/\C)$ does not split over $\A$.
Suppose now, towards a contradiction,
that $\t(a/\B)$ splits over $\A$, and let
$b, c \in \B \subseteq \C$ witness the splitting.
By the minimality of the type $p$, 
we have $\t(a'/\B)=\t(a/\B)$,
and thus $\t(a'b/\A) \neq \t(a'c/\A)$,
so $\t(a'/\C)$ splits after all, a contradiction.
\end{proof}

\begin{corollary}\label{furthermore}
Let $\a \ge \o$ be an ordinal,
and let $\B_i \in \K^*$ and $b_i \in \B_{i+1} \setminus \B_i$, $i \le \a$,
be such that $\B_0=\A$,
$\B_i \subseteq \B_j$ whenever $i<j$,
$\t(b_i/\A)=p$, and $\B_{i+1}=\langle \B_i, b_i \rangle$.
Then, the sequence $(b_i)_{i<\a}$ 
is indiscernible over $\A$. 
\end{corollary}

\begin{proof}
By Lemma \ref{ordind}, it is enough to show that the sequence $(b_i)_{i<\a}$
is order indiscernible.
This follows from Lemma \ref{dnsplit} as usually.
\end{proof}

\begin{lemma}\label{pregeometry}
Suppose $\A \subseteq \B$, $\B \in \K^*_{>\l^+}$, $X=\{a \in \B \, | \, \t(a/\A)=p\}$, and for $A \subseteq X$, let $cl(A)= \langle \A \cup A \rangle \cap X$.
Then, $(X, cl)$ is a pregeometry. 
\end{lemma}
\begin{proof}
For exchange, we may without loss assume $\l^{+} < \v \B  \v \le \k$, 
since if exchange fails, it fails also in a small submodel.
assume $c \in X$ is a tuple, and $a \in cl(cb) \setminus cl(c)$.
Suppose for the sake of contradiction that $b \notin cl(ca)$.
By Lemma \ref{exists} and Corollary \ref{saturated2}, there is a sequence  $(a_i)_{i<\o} \in X$ such that $a_i \notin \langle \A, c, a_j \rangle_{j<i}$, $a_0=a$, and $a_1=b$.
By Corollary  \ref{furthermore}, this sequence is indiscernible, so there is some automorphism of $\langle \A, c, a_i \rangle_{i<\o}$ 
which fixes $\langle \A, c \rangle$ and swaps $a$ and $b$.
It follows that $a \notin cl(cb)$, a contradiction.
 \end{proof}

\begin{lemma}\label{generated}
Suppose $\A \subseteq \B \in \K^*_{>\l^{+}}$ and $(a_i)_{i<\a} \subseteq \B$ is a maximal independent sequence of realisations of $p$.
Then, $\B=\langle \A, a_i \rangle_{i<\a}$.
\end{lemma}

\begin{proof}
If the statement does not hold, then there is a counterexample of power at most $\k$,
so we may assume $\v \B \v=\b \le \k$.
Using Lemma \ref{exists}, we find a model
$\C \supseteq \A$ such that $(c_i)_{i<\b} \subseteq \C$ is an independent sequence of realisations of $p$, and $\C=\langle \A, c_i \rangle_{i<\b}$.
By saturation (Corollary \ref{saturated2}), we may assume $\C=\B$.
Since $(a_i)_{i<\a}$ was maximal, we have $(c_i)_{i<\b} \subseteq cl(\{a_i\}_{i<\a})$,
and thus $\B=\langle \A, c_i \rangle_{i<\b} \subseteq \langle \A, a_i \rangle_{i<\a}$.
\end{proof}

\begin{theorem}\label{categoricity}
Let $\K$ be a universal class with arbitrarily large models and let $\k$ be a regular cardinal such that $\K$ 
is categorical in $\k$.
Suppose $\l=LS(\K)$ and $\k>\l^+$. 
Let $\K^*$ be the class consisting of all the models $\B \in \K$ such that $\vert \B \vert > \l$ 
and $\B$ can be embedded in some $\C \in \K_{\ge \k}$.
Then, $\K^*$ is categorical in every $\mu>\l^+$.
\end{theorem}

\begin{proof}
Suppose $\v \B \v = \v \C \v = \mu>\l^+$.
By Lemma \ref{universal}, we may assume $\A \subseteq \B, \C$, and by Lemma \ref{generated}, 
there are independent sequences  $(b_i)_{i<\mu} \subseteq \B$ and $(c_i)_{i<\mu} \subseteq \C$ such that the $b_i$ and $c_i$ are realisations of $p$ and $\B=\langle \A, b_i \rangle_{i<\mu}$ and  $\C=\langle \A, c_i \rangle_{i<\mu}$.
The map $f: \B \to \C$ defined by $f \raj \A=id$ and $f(b_i)=c_i$ for $i<\mu$ 
generates an isomorphism since the type $p$ is minimal. 
\end{proof}

Next we show that only small models in $\K$ fall outside $\K^*$.

\begin{theorem}\label{jatko}
If $\B \in \K$ and $\v \B \v \ge \beth_{(2^{\l^+})^+}$, then $\B \in \K^*$.
\end{theorem}

\begin{proof}
Denote $H= \beth_{(2^{\l^+})^+}$.
If $\k\le H$, we are done, so assume $H<\k$.
Suppose $\B \in \K$, $\v \B \v \ge H$ and $\B \notin \K^*$.
By Lemma \ref{peruslemma} (ii), there is some $\C \subseteq \B$ such that $\v \C \v = \l^+$
and $\C \notin \K^*$.
With an Ehrenfeucht-Mostowski construction, we find a sequence $I=(a_i)_{i<\k}$ quantifier free indiscernible over $\C$ so that for all finite $X \subset \k$, it holds that $\langle \C, a_i \rangle_{i \in X} \in \K$.
Denote $\D=\langle \C, a_i \rangle_{i<\k}$.
Since $\K$ is an AEC, $\D \in \K$.
Now we have $\C \subseteq \D \in \K_\k$, and thus $\C \in \K^*$, a contradiction. 
\end{proof}

\section{Models are vector spaces}\label{vector}

\noindent
In this section, we will prove two theorems that together state that 
up to a coordinatisation all models in the class $\K^*_{>\l^+}$ either have the structure of a vector space or a trivial pregeometry.
Recall that in the previous section, after Lemma \ref{uniqext}, we fixed a model
 $\A \in \K^*_{\l^+}$ and a minimal type $p \in S_{qf}(\A)$ (see Definition \ref{minimal}).
 
Suppose $\B \in \K^*_{>\l^+}$, $\A \subseteq \B$ and $X=\{a \in \B \, | \, \t(a/\A)=p\}$.
For $A \subseteq X$, let $cl(A)= \langle \A \cup A \rangle \cap X$.
By Lemma \ref{pregeometry}, $(X, cl)$ is a pregeometry. 
We will use a result due to Zilber to prove that either this pregeometry is trivial or the set $X$ can be given the structure of a vector space, and then show that the structure $\B$ can be strongly coordinatised using elements from $X$.

To be able to prove this, we need the following two lemmas.
The first one states that the class $\K^*$ has AP and that all the models in $\K^*$ are saturated,
and the second one says that non-algebraic quantifier free types over models of size $\l^+$ always have non-algebraic extensions.
 
\begin{lemma}\label{apsat}
The class $\K^*$ has AP and every model in $\K^*$ is saturated.  
\end{lemma}

\begin{proof}
Let $\xi > \k$ be a regular cardinal.
It suffices to prove the lemma for $\K^*_{<\xi}$.
Let $\K^{**}$ be the class consisting of all the models $\B \in \K$ such that $\v \B \v > \l$ and there is some $\C \in \K$ such that $\B \subseteq \C$ and $\v \C \v \ge \xi$.

By Theorem \ref{categoricity}, $\K^{**}=\K^*$. 
Since $\K^{**}$ and $\xi$ satisfy the assumptions posed for $\K^*$ and $\k$ in the beginning of Section \ref{cat} (note that $\K$ is $\xi$-categorical by Theorem \ref{categoricity}),
the class $\K^{**}_{<\xi}$ has AP by Lemma \ref{AP}, and every model in $\K^{**}_{<\xi}$ is saturated by Corollary \ref{saturated}. 
 \end{proof}

\begin{lemma}\label{tyyppilaajenee}
Let $\B, \C \in \K^*$ be such that $\v \B \v =\l^+$ and $\B \subseteq \A$, and let $b \in \C \setminus \B$.
Then, there is some $\D \in \K^*$ and $c \in \D \setminus \C$ such that $\C \subseteq \D$ and
$\t(c/\B)=\t(b/\B)$.
\end{lemma}

\begin{proof}
Choose an increasing chain of models $\B_i \subseteq \B$, $i < \l^+$, such that $\v \B_i \v = \l$ and $\bigcup_{i<\l^+} \B_i=\B$.
By Lemma \ref{weaksat2}, $\B$ is weakly saturated, and thus there is  some $a_i \in \B \setminus \B_i$ such that $\t(a_i/\B_i)=\t(b/\B_i)$ (note that $\B_i \notin \K^*$, but this is not a problem since weak saturation is defined for arbitrary sets, see Definition \ref{satdef}). 
Moreover, we may choose these elements so that if $i \neq j$, then $a_i \neq a_j$.
Indeed, for each $i$, there is some $k<\l^+$ (note that we might have $k>i$) such that $a_j \in \B_k$ for each $j<i$,
and we can choose the element $a_i$ so that $a_i \in \B \setminus \B_k$.
 
By Lemma \ref{delta0}, we may assume the sequence $I=(a_i)_{i<\l^+}$ is quantifier free indiscernible, 
and by Lemma \ref{avexists0}, there is some model $\D \in \K^*$ and some $c \in \D$ such that $\C \subseteq \D$ and $\t(c/\C)=Av(I/\C)$.
By construction, we have $\t(b/\B)=Av(I/\B)$, and thus $\D$ and $c$ are as wanted.

\end{proof}

We will apply a result of Zilber's to show that either the pregeometry $(X, cl)$ is trivial or the set $X$ can be given the structure of a vector space.
For this, we need to show that $X$ can be viewed as a quasi-Urbanik structure in the sense of the below definition from \cite{zilber}.


\begin{definition}\label{hertrans}
Let $G$ be a permutation group of a set $X$.
For arbitrary $Y \subseteq X$, we denote by $G_Y$ the subgroup of $G$ that fixes each element of $Y$ and by $[Y]$
the set of elements in $X$ that are fixed by each element of $G_Y$.
We say $G$ is \emph{hereditarily transitive} if for each finite $Y \subseteq X$, the subgroup $G_Y$ is transitive on $X \setminus [Y]$, i.e. for any $x_1, x_2 \in X \setminus [Y]$, there is some $g \in G_Y$ such that $g(x_1)=x_2$.
\end{definition}

\begin{definition}
A structure is \emph{quasi-Urbanik} if its group of automorphisms is hereditarily transitive.
\end{definition}

Recall that we have fixed a model $\A \in \K^*_{\l^+}$ and a minimal type $p \in S_{qf}(\A)$.
 
\begin{theorem}\label{quasiurb}
Suppose $\B \in \K^*_{>\l^+}$ is such that $\A \subseteq \B$, let $X=\{a \in \B \, | \, \t(a/\A)=p\}$, and for $A \subseteq X$, let $cl(A)=\langle \A \cup A \rangle \cap X$.
Then, $(X, cl)$ is a pregeometry, and either it is trivial or has the structure of a vector space. 
\end{theorem}

\begin{proof}
By Lemma \ref{pregeometry}, $(X, cl)$ is a pregeometry.

Let $L^*$ be the language consisting of predicates $R_q$ for the types $q \in S(\A)$ satisfied by tuples from $X$.
Consider $X$ as an $L^*$-structure by setting $x \in R_q$ if and only if $\t(x/\A)=q$.
The automorphisms of this structure are exactly the restrictions to $X$ of those automorphisms of $\B$ that fix $\A$ pointwise.

We claim that $X$ is quasi-Urbanik (as an $L^*$-structure).
Let $Y \subseteq X$ be a finite subset, and denote by $G_Y$ the automorphisms of $X$ that fix $Y$ pointwise.
Using the notation of Definition \ref{hertrans}, we now have $[Y]=cl(Y)$. 
If $x, y \in X \setminus cl(Y)$, then there is some automorphism of $\B$ fixing $cl(Y) \cup \A$ 
and sending $x$ to $y$.
It restricts to an automorphism of $X$ fixing $cl(Y)=[Y]$.
Thus, the $L^*$-structure $X$ is quasi-Urbanic.
 
By \cite{zilber}, Theorem B (on p. 167), there is a pregeometry $(V, cl)$ and a bijection between $X$ and $V \setminus cl(\emptyset)$ such that if $X_0 \subseteq X$ maps to $V_0 \subseteq V$, then $cl(X_0)$ maps to $cl(V_0) \setminus \emptyset$. 
Moreover, one of the following holds for $(V,cl)$: 
\begin{enumerate}[(i)]
\item The set $V$ is a vector space with a distinguished subspace $W$, and if $V_0 \subseteq V$, then the pregeometry is given by $cl(V_0)=span(W, V_0)$; 
\item the set $V$ is an affine space with a distinguished  linear subspace of parallel translations $W$, and if $V_0 \subseteq V$, then $cl(V_0)=aff_W(V_0)$ (the affine space generated by $V_0$ together with the action of $W$); 
\item there is some group $H$ acting on $V$, and if $V_0 \subseteq V$, then $cl(V_0)=\{h(v) \, | \, h \in H, v \in V_0\}$.
\end{enumerate}
In cases (i) and (ii), $(X, cl)$ has the structure of a vector space,
and in case (iii), the pregeometry is trivial. 
\end{proof}


Note that all models in the class $\K^*$ are universal, 
and thus each  $\B \in \K^*_{>\l^+}$ contains an isomorphic copy of $\A$.
Hence, Theorem \ref{quasiurb} describes the structure of all models in $\K^*_{>\l^+}$
eventhough its statement contains the assumption 
$\A \subseteq \B$.
Next, we prove a theorem saying that if 
$\B \in \K^*_{>\l^+}$ and $\A \subseteq \B$,
then there is a strong coordinatisation of $\B$ using the elements of $X=p(\B)$.
These two theorems together show that all models in $\K^*_{>\l^+}$ either are essentially vector spaces or trivial.

Let $\A^* \in \K^*_{\l^+}$
be a model such that $\A^*= \langle \A, a_i \rangle_{i<\o}$,
where $(a_i)_{i<\o}$ is an independent (with respect to the pregeometry given in the statement of Theorem \ref{quasiurb}) sequence of realisations of the type $p$.
Denote $A=(a_i)_{i<\o}$. 
Now, $A$ is a pregeometry basis for $p(\A^*)$. 
Let $q \in S(\A^*)$ be the unique non-algebraic extension of $p$ to $\A^*$.
For each model $\B \in \K^*_{>\l^+}$ such that $\A^* \subseteq \B$,
we can define a pregeometry on $X^*=q(\B)$ 
similarly as we have done for $p(\B)$:
for a set $Y \subseteq X^*$, we let $cl(Y)=\langle \A^*, Y \rangle \cap X^*$.
Note that this closure operator is obtained by localising the pregeometry of $p(\B)$ 
at $p(\B) \cap \A^*$.

In the statement of the next theorem, we assume that $\B \in \K^*_{>\l^+}$ and $\A^* \subseteq \B$,
but again, each model in $\K^*_{>\l^+}$ 
contains a submodel isomorphic to $\A^*$,
and hence the theorem can be seen hold for all $\B \in \K^*_{>\l^+}$.

By the notation $b \in dcl(B)$ we mean that there is a quantifier free first order formula $\phi(x,y)$
and some $a \in B$ such that $b$ is the unique element satisfying $\phi(x,a)$.
If $a \in dcl(B,b)$ and $b \in dcl(B,a)$, we say that $a$ and $b$ are \emph{interdefinable} over $B$.

\begin{theorem}\label{koordinaatit}
Let $\B \in \K^*_{>\l^+}$, suppose $\A^* \subseteq \B$, and let $b \in \B$.
Setting $X^*=\{a \in \B \, | \, \t(a/\A^*)=q\}$, there are elements
$a_1, \ldots, a_n \in X^*$ for some $n<\o$, such that $b$ is interdefinable with the tuple $(a_1, \ldots, a_n)$ over $\A^*$.
\end{theorem}

\begin{proof} 

We prove first an auxiliary claim.
Recall that $A$ was chosen to be a basis for $\A^*$
in $p(\B)$.


\begin{claim}\label{uniquesmallest}
If there are some closed sets $Y, Z \subseteq X^*$ such that $b \in \langle \A^*, Y \rangle \cap \langle \A^*, Z \rangle$,
then $b \in \langle \A^*, Y \cap Z \rangle$.
\end{claim}

\begin{proof} 
By Theorem \ref{quasiurb},
$p(\B)$ is locally modular,
and thus $q(\B)=X^*$ is modular,
so $Y$ and $Z$ are independent over $Y \cap Z$,
and we can choose bases $Y_0$ and $Z_0$ for $Y$ and $Z$, respectively,
so that $Y_0 \cup Z_0$ is a basis for $Y \cup Z$
(take a basis for $Y \cap Z$, enlarge it first to a basis for $Y$ and then to a basis for $Y \cup Z$).
Then, $A  \cup Y_0 \cup Z_0$ is an independent sequence in $p(\B)$. 

  
Now, there are some $a \in \A$ and some $a' \in A$ such that
$$b \in \langle a, a', Y_0 \rangle \cap \langle a, a', Z_0 \rangle.$$
Since set $A \cup Y_0 \cup Z_0$ is independent, there is an automorphism $f$ of $\B$ fixing the set $\A a'Z_0$ pointwise, such that $f(Y_0 \setminus Z_0) \subseteq A$.
Since $b \in \langle \A, a', Z_0 \rangle$, we have $f(b)=b$.
On the other hand, we get $f(b) \in dcl (\A, a', f(Y_0)) \subseteq dcl(\A^*, Y_0 \cap Z_0)$. 
\end{proof}

By Claim \ref{uniquesmallest}, there is a unique smallest closed set $X_b \subseteq X^*$ such that $b \in \langle \A^*, X_b \rangle$.
Let $\{a_1, \ldots, a_n \}$ be a basis for $X_b$.
We claim that it is as wanted.
Clearly $b \in dcl(\A^*, a_1, \ldots, a_n)$.
We will show that $a_i \in \langle \A^* b \rangle$, and the claim will follow.

Suppose $1 \le i \le n$ and $a_i \notin \langle \A^* b \rangle.$
Let $f$ be an automorphism of $\B$ that fixes the set $\A^*b$ pointwise.
Since $X_b$ is unique, $f$ fixes it as a set, and thus $f(a_i) \in X_b$ for $1 \le i \le n$.
Because $\v X_b \v=\l^+$, the type $\t(a_i/ \langle \A^* b \rangle)$ has at most $\l^+$ many realisations in $\B$.
We will derive a contradiction by constructing a sequence of $\l^{++}$ many distinct realisations.
Denote $\C_0=\B$, and choose models $\C_j$ and elements $c_j \in \B$ for $j<\l^{++}$ as follows.
If $j$ is a limit ordinal, take $\C_j=\bigcup_{k<j} \C_k$.
By Lemma \ref{tyyppilaajenee}, there is some $\C_{j+1} \in \K^*$ and $c \in \C_{j+1} \setminus \C_j$
such that $\C_j \subseteq \C_{j+1}$ and $\t(c/\langle \A^*, b \rangle) =\t(a_i/\langle \A^*, b \rangle)$.
Using Lemma \ref{apsat}, we find an element $c_j \in \B$ such that $\t(c_j / \langle \A^*, b, c_k \rangle_{k<j})=\t(c/\langle \A^*, b, c_k \rangle_{k<j})$.
 
 \end{proof}
 
 We now apply the above theorem to show that $\K^*$ is categorical in $\l^+$
 and thus totally categorical.
 
 \begin{theorem}\label{catbonus}
 Let $\K$ be a universal class with arbitrarily large models and let $\k$ be a regular cardinal such that $\K$ 
is categorical in $\k$.
Suppose $\l=LS(\K)$ and $\k>\l^+$. 
Let $\K^*$ be the class consisting of all the models $\B \in \K$ such that $\vert \B \vert > \l$ 
and $\B$ can be embedded in some $\C \in \K_{\ge \k}$.
Then, $\K^*$ is totally categorical.
 \end{theorem}
 
 \begin{proof}
 By Theorem \ref{categoricity}, the class $\K^*$ is categorical in every $\mu>\l^+$.
 We show it is categorical also in $\l^+$.
 
Let $\B \in K^{*}_{\l^{++}}$ be such that $\A^{*}\subseteq\B$,
and let $(a_{i})_{i<\l^{++}}$ be a basis of $q(\B )$.
Let $\A^{**}=\langle \A^{*},a_{i} \rangle_{i<\l^{+}}$, and suppose
$\C\in K^{*}_{\l^{+}}$. 
It is enough to prove that
$\C\cong\A^{**}$.

By Lemma \ref{universal}, $\C$ is universal, so we may assume 
$\A^{**}\subseteq\C$, and since $\B$ is saturated by Corollary \ref{saturated2}, we may assume
 $\C\subseteq\B$. 
Let $(b_{i})_{i<\a}$ be a basis of  $q(\C)$ such that $b_{i}=a_{i}$ for
$i<\l^{+}$, and extend it to a basis
$(b_{i})_{i<\l^{++}}$ of $q(\B )$.
Since $(b_{i})_{i<\l^{+}}$ and $(b_{i})_{i<\a}$ are
quantifier free indiscernible sequences over $\A^{*}$ of the same cardinality
and the former  is a subset of the latter,
it is enough to prove that $\C=\langle \A^{*},b_{i} \rangle_{i<\a}$.

For this, let $a\in\C$. 
By Theorem \ref{koordinaatit},
there are
$c_{1}, \ldots ,c_{n}\in q(\B )$ such that
$a\in \langle \A^{*},c_{1}, \ldots,c_{n} \rangle$ and
for all $1\le i\le n$, $c_{i}\in \langle \A^{*},a \rangle$.
Since $\langle \A^{*},a \rangle \subseteq\C$,  we have $c_{1}, \ldots ,c_{n}\in q(\C )$
and thus $a\in \langle \A^{*},b_{i} \rangle_{i<\a}$. 
\end{proof}
 
We finish this paper with a remark on the definition of universal classes.
As Kirby pointed out to us, the definition is very syntactic: if we replace the functions by their graphs in the models of a universal class,
then the class is not in general universal any more.
This raises the question whether our results can be generalised to less syntactic cases.
The answer is, of course, positive, and we now describe one easy case.
 
Suppose $(\K',\subseteq )$ is an AEC with arbitrary large models,
$\l=LS(\K')$, and $\K'$ is categorical in some regular cardinal $\k >\l^{+}$.
Moreover, we require that the class is closed under definably closed submodels,
and for this, we need to redefine the notion of definably closed to make it work in a setup that does not assume AP
(note that the following definition coincides with the previous one when we are working over the model $\A^*$ defined as above). 
 
\begin{definition}
If $\A \in \K'$, $b \in \A$ is a singleton and $a \in \A^n$, 
we say that
$b$ is \emph{definable} over $a$ (in $\A$) if there is a quantifier free formula
$\phi (x,y)$ such that 
\begin{itemize}
\item $\phi (\A ,a)=\{ b\}$;
\item if $\C \in \K'$ and $a' \in \C^n$ is such that
$t_{qf}(a'/\emptyset )=t_{qf}(a/\emptyset )$,
then $\phi (\C ,a')$ is a singleton.
\end{itemize}
\end{definition}

We formulate our requirement that $\K'$ be closed under definably closed submodels by assuming that if
$\B \in \K'$, $\A \subseteq \B$, and the condition (*) below holds, then $\A \in \K'$.

\vspace{0.2cm} 
(*) If $b\in\B$ is definable over some $a\in\A^{n}$,
then $b\in\A$.
\vspace{0.2cm}

To show that our results hold for $\K'$,
we define another class $\K$ as follows.  
For each quantifier free formula $\phi =\phi (x,y)$, 
we add a function symbol $f_{\phi}$
to the vocabulary of $\K'$.
If  $\A\in \K'$, we define a model $\A^{*}$ by adding
the following interpretations for $f_{\phi}$:
if $a=(a_{1},...,a_{n})\in\A^{n}$, the singleton $b$ is definable over $a$
and $\phi$ witnesses this, then
$f_{\phi}(a)=b$ and otherwise $f_{\phi}(a)=a_{1}$.
Let the class $\K$ consist of the models $\A^*$
for all $\A \in \K'$.
Now it is easy to see that
$(\K,\subseteq )$ is a universal class, $LS(\K)=\l$ and
$\K$ is categorical in $\k$. 
Moreover, the map 
$\A\mapsto\A^{*}$
is a bijection from $\K'$ to $\K$ and
if $\A, \B \in \K'$, then
$\A^{*}\subseteq\B^{*}$ if and only if $\A \subseteq \B$
and $\A^{*}\cong\B^{*}$ if and only if $\A\cong\B$.
Since the assumptions we gave at the beginning of this paper hold for $\K$,
our results apply there,
and they can be straightforwardly transferred to $\K'$.

\end{document}